\documentclass[graybox]{svmult}
\usepackage{makeidx}         
\usepackage{graphicx}        
\usepackage{multicol}        
\usepackage[bottom]{footmisc}

\usepackage{newtxtext}       %
\usepackage{newtxmath}       

\usepackage{cite}

\usepackage{enumitem}
\usepackage{algorithmic}
\usepackage{algorithm}
\usepackage{euscript}
\usepackage{array}
\usepackage[all,cmtip]{xy}
\usepackage{float}
\usepackage{mathtools}
\usepackage{todonotes}
\usepackage{booktabs}

\usepackage{PaperDef}

\newcommand {\mat}[1]{\left[\begin{array}{#1}}
\newcommand {\rix}          {\end{array}\right]}
\newcommand {\eq}       [1] {\begin{equation}\label{#1}}
\newcommand {\en}           {\end{equation}}
\newcommand {\mpar} [1] {\marginpar{\fussy\tiny #1}}
\renewcommand {\mpar} [1] {}

\newcommand {\rank}       {\mathop{\rm rank}\nolimits}

\begin{document}


 \title*{Structure-preserving Interpolatory Model Reduction
 for Port-Hamiltonian Differential-Algebraic Systems}
\titlerunning{Structure-preserving Interpolatory Model Reduction for \textsf{pHDAE}s}
\author{C.~A. Beattie \and S. Gugercin \and V. Mehrmann}

\institute{C.~A. Beattie \at Department of Mathematics, Virginia Tech, Blacksburg, VA 24061-0123, USA, \\ \email{beattie@vt.edu}
\and S. Gugercin \at Department of Mathematics, Virginia Tech, Blacksburg, VA 24061-0123, USA,\\  \email{gugercin@vt.edu}
\and V. Mehrmann \at Institut f\"{u}r Mathematik, Technische Universit\"{a}t Berlin, D-10829, Berlin, Germany \\  \email{mehrmann@math.tu-berlin.de }
}
\maketitle

\vspace{-2.5cm}
\begin{center}
\emph{Dedicated to Thanos Antoulas on the occasion of his 70th birthday}
\end{center}
\vspace{1ex}

\abstract{
We examine interpolatory model reduction methods that are well-suited for treating large scale port-Hamiltonian differential-algebraic systems in a way that is able to preserve and take advantage of the underlying structural features of the system.   We introduce approaches that
 incorporate regularization together with a prudent selection of interpolation data. We focus on linear time-invariant systems and
 present a systematic treatment of a variety of model classes that include combinations of index-$1$ and index-$2$ systems,  describing in particular how constraints may be represented in the transfer function so that the polynomial part can be preserved with interpolatory methods. We propose an algorithm to generate effective interpolation data and illustrate its effectiveness on a 
 numerical example.
}

\vspace{-1ex}
\section{Introduction}\label{sec:intro}
Port-Hamiltonian ( \textsf{pH} ) systems are network-based models that arise within a modeling framework in which a physical system model is decomposed into hierarchies of submodels that are interconnected principally through the exchange of energy.  At some level the submodels typically reflect one of a variety of core modeling paradigms that describe  phenomenological aspects of the dynamics having different physical character, such as e.g. electrical, thermodynamic, or mechanical.   The  \textsf{pH}  framework is able to knit together submodels featuring  dramatically different physics through a disciplined focus on energy flux as a principal mode of system interconnection;  \textsf{pH}  structure is inherited via power conserving interconnection, and a variety of physical properties and functional constraints (e.g., 
passivity and energy and momentum conservation) are encoded directly into the structure of the model equations \cite{BeaMXZ18,Sch13}.  Interconnection of submodels often creates further constraints on system behavior and evolution, originating as conservation laws (e.g., Kirchoff's laws or mass balance), or as position and velocity limitations in mechanical systems.
As a result system  models
are often naturally posed as combinations of dynamical system equations and algebraic constraint equations, i.e. as 
\emph{Port-Hamiltonian descriptor systems} or \emph{port-Hamiltonian differential-algebraic equations} (\textsf{pHDAE}).

When a \textsf{pHDAE}  system is linearized around a stationary solution, one obtains a linear time-invariant \textsf{pHDAE}  with specially structured coefficient matrices, see \cite{BeaMXZ18}.
Although the approach we develop here can be extended easily to more general settings,
we narrow our focus to a particular formulation 
as one of the simpler among alternative formulations
laid out in \cite{BeaMXZ18,MehMW18}.
\begin{definition}\label{pHDAEDEF} A linear time-invariant \textsf{DAE} system of the form
\begin{equation}\label{pHDAE}
\begin{aligned}
\bE\dot{\bx}&=(\bJ-\bR)\, \bx+(\bB-\bP)\,\bu,\quad  \bx(t_0)=0,\\
\by&=(\bB+\bP)^T\,\bx+(\bS+\bN)\,\bu,
\end{aligned}
\end{equation}
with $\bE$, $\bJ$, $\bR\in\IR^{n \times n}$, $\bB, \bP\in\IR^{n \times m}$, $\bS=\bS^T$, $\bN=-\bN^T\in\IR^{m \times m}$, on a compact interval $\mathbb I\subset \IR$ is a \textsf{pHDAE} system if the following properties are satisfied:
\begin{enumerate}
\item The differential-algebraic operator
\begin{align*}
\bE\frac{d}{dt}-\bJ :   {\mathcal C}^1(\mathbb{I},\IR^n)\rightarrow  {\mathcal C}^0(\mathbb{I},\mathbb{R}^n)
\end{align*}
is \emph{skew-adjoint}, i.e., we have that
$\bJ^T=-\bJ$ and $\bE=\bE^T$,
\item $\bE$ is positive semidefinite, i.e., $\bE\geq 0$, and
\item the \emph{passivity} matrix

\vspace{-4ex}
\begin{align*}
\bW=\begin{bmatrix}\bR & \bP\\ \bP^T & \bS\end{bmatrix}\in\mathbb{R}^{(n+m)\times(n+m)}
\end{align*}
is symmetric positive semi-definite, i.e., $\bW=\bW^T\geq 0$.
\end{enumerate}
The associated quadratic \emph{Hamiltonian function} ${\mathcal H}:\mathbb{R}^n\rightarrow\mathbb{R}$ of the system is 
${\mathcal H}(\bx)=\frac{1}{2}\bx^T\bE\bx.$
\end{definition}
In Definition~\ref{pHDAEDEF}, the Hessian matrix of ${\mathcal H}(\bx)$ is the matrix $\bE$ which is also referred to as \emph{energy matrix};
$\bR$ is the \emph{dissipation matrix}; $\bJ$ is the \emph{structure matrix} describing the energy flux among internal energy storage elements; and $\bB  \pm \bP$ are \emph{port matrices} describing how energy enters and leaves the system.  $\bS$ and $\bN$ are matrices associated with a \emph{direct feed-through} from input $\bu$ to output $\by$.
If the system has a 
solution $\bx\in \mathcal{C}^1(\mathbb{I},\mathbb{R}^n)$ in $\mathbb{I}$ when for 
 a given input function $\bu$, then the dissipation inequality
\begin{align*}
\frac{d}{dt}{\mathcal H}(\bx)=\bu^T\by-\begin{bmatrix}\bx\\\bu\end{bmatrix}^T\bW\begin{bmatrix}\bx\\\bu\end{bmatrix}
\end{align*}
must hold, i.e.,  the system is both \emph{passive} and \emph{Lyapunov stable}, as $\mathcal H$ defines a \emph{Lyapunov function}, see \cite{BeaMXZ18}.

In practice, network-based automated modeling tools such as
{\sc modelica}, {\sc Matlab/Simu\-link}, or
\textsf{20-sim}\footnote{\texttt{https://www.modelica.org/}, \texttt{http://www.mathworks.com}, \texttt{https://www.20sim.com}} each may produce system models that are overdetermined as a consequence of redundant modeling.
Thus, the automated generation of a system model usually must be followed by a reformulation and regularization procedure (that preserves the \textsf{pH}  structure and  makes the system model compatible with standard simulation, control and optimization tools.

\textsf{pHDAE} models may be very large and complex, e.g., models that arise from semi-discretized (in space) continuum models in hydrodynamics \cite{EggKLMM18,
HeiSS08, Sty06}, or mechanics \cite{
MehS05}.  In such cases, model reduction techniques are necessary to apply control and optimization methods. Preservation of the \textsf{pH}, the constraint, and the interconnection  structure is necessary to maintain model integrity.
For liner time -invariant \textsf{pH}  systems with 
positive-definite $\bE$, such
methods are well-developed, including tangential interpolation approaches \cite{GugPBS12,
GugPBS09}, moment matching \cite{PolS10,
PolS11,WolLEK10} as well as effort and flow constraint reduction methods \cite{PolS12}.  Interpolatory approaches have been extended to nonlinear  \textsf{pH}  systems as well, in \cite{BG2011pHstructPresMOR} and \cite{ChaBG16}.

In the case of singular $\bE$, only recently in \cite{EggKLMM18,HauMM19_ppt} first model reduction techniques have been introduced that preserve the \textsf{pH}  structure, while for unstructured DAEs  such constraint preserving methods  have been introduced in
\cite{GugSW13,HeiSS08,MehS05,Sty06}.

We discuss structure preserving model reduction methods incorporating both regularization and interpolation for linear \textsf{pHDAE}  systems having the form (\ref{pHDAE}). For different model classes we describe how  constraints are represented in the transfer function, and how the polynomial part can be preserved with interpolatory model reduction methods.  The key step identifies, as in \cite{BeaMXZ18,MehM19_ppt
}, all the redundancies as well as both explicit and implicit system constraints, partitioning the system equations into redundant, algebraic, and dynamic parts.  Only the dynamic part is then reduced in a way that preserves the structure.

In \S\ref{sec:regularization}, we consider \textsf{pHDAE}s of the form \eqref{pHDAE}
and note important simplifications of the general regularization procedure.
Structure preserving model reduction of \textsf{pHDAE}s  using interpolatory methods is
discussed first in general terms in  \S\ref{sec:intmodred}, and then in more detail for several specific model structures.
We propose an algorithm to generate effective interpolation data in \S\ref{sec:imp} followed by a numerical example that demonstrates its effectiveness.

\section{General Differential-Algebraic Systems} \label{sec:dess}
In this section we recall some basic properties  of general linear constant coefficient \textsf{DAE} systems
\begin{equation}\label{DAE}
\begin{aligned}
                \bE \dot{\bx} & = \bA \bx + \bB \bu,\ \bx(t_0)=0,\\
                \by         & = \bC \bx+  \bD \bu,
\end{aligned}
\end{equation}
where
\(\bE,\,\bA \in \IR^{n \times n}\),
\(\bB \in \IR^{n \times m}\),
\(\bC \in \IR^{p \times n}\),
\(\bD \in \IR^{p \times m}\); for details, see, e.g., \cite{BunBMN99}.

The matrix pencil \(\lambda \bE -\bA\) is
said to be \emph{regular} if
\(\det(\lambda \bE - \bA) \ne 0\) for some
\(\lambda  \in \IC\). For regular pencils, the \emph{finite eigenvalues} are
the values $\lambda\in \IC$ for which \(\det(\lambda \bE - \bA ) = 0\).
If the \emph{reversed pencil} \(\lambda \bA - \bE \) has the eigenvalue $0$ then  this is the \emph{infinite eigenvalue} of $\lambda \bE - \bA $.

With  zero initial conditions $\bx(t_0)=0$ as in \eqref{DAE} and a regular pencil \(\lambda \bE - \bA\), we obtain, in the frequency domain, the rational \emph{transfer function}
$ \bH(s)=\bC(s\bE-\bA)^{-1}\bB+\bD$,
which maps the Laplace transform of the input functions $\bu$ to the Laplace transform of the output function $\by$.  If the transfer function is  written as
\[
\bH(s)=\bG(s)+\bP(s),
\]
where $\bG(s)$ is a proper rational matrix function and $\bP(s)$ is a polynomial matrix function, then the finite eigenvalues are the poles
of the proper rational part, while infinite eigenvalues belong to the polynomial part.

In the following, a matrix with
orthonormal columns spanning the right nullspace of the matrix $\bM$
is denoted by $S_{\infty}(\bM)$ and a matrix with orthonormal
columns spanning the left nullspace of $\bM$ by $T_{\infty}(\bM)$.

Regular pencils can be analyzed  via the Weierstra\ss\   Canonical Form;
see, e.g., \cite{Gan59b}.
The {\em index} $\nu$ of the pencil $\lambda \bE-\bA$ is the size of the largest block associated with the eigenvalue $\infty$
in the Weierstra\ss\ canonical form
and if \(\bE\) is nonsingular, then the pencil has index $0$.

To analyze general \textsf{DAE} systems and also to understand the properties of the transfer function, some controllability and observability conditions are needed, see e.g. \cite{BunBMN99,Dai89}. A system that satisfies
${\rm C1}:\ \rank [ \lambda \bE - \bA,\, \bB ] = n \textrm{ for all } \lambda \in \IC$
and
${\rm C2}:\          \rank [ \bE,\, \bA S_\infty(\bE),\, \bB ] = n$
is called \emph{strongly controllable}\cite{BunBMN99}.  If a system satisfies analogous observability conditions 
${\rm O1}:\  \rank\mat{cc}\lambda \bE^T-\bA^T & \bC^T\rix = n \textrm{ for all }\lambda \in \IC$, and
${\rm O2}:\ \rank \mat{ccc} \bE^T & \bA^T T_{\infty}(\bE)  & \bC^T \rix = n$,
then it is called \emph{strongly observable}.
If a system is strongly controllable and strongly observable then it is called \emph{minimal}.

Conditions ({C1}) and ({C2}) are preserved under non-singular
equivalence transformations as well as under state and output feedback.
Note, however, that regularity or non-regularity of the pencil, the index, and the polynomial part of the transfer function are in general not preserved under state or output feedback. For systems that satisfy~({C2}), there exists a suitable linear state feedback matrix $\bF_1$ that such that $\lambda \bE-(\bA+\bB \bF_1))$ is regular and of index at most one. Also if conditions~({C2}) and~({O2}) hold, then there exists a linear output feedback matrix $\bF_2$ so that the pencil $\lambda \bE-(\bA+\bB\bF_2 \bC)$ has this property, see \cite{BunBMN99}.

\section{Regularization of \textsf{PHDAE} Systems}\label{sec:regularization}
In general, it cannot be guaranteed that a system, generated either from a realization procedure or an automated modeling procedure, has a regular pencil $\lambda \bE-\bA$. Therefore, typically a regularization procedure has to be applied, see \cite{BinMMS15_ppt,CamKM12,KunM06}.
For \textsf{pHDAE}s  the structure helps simplify theses general regularization procedures significantly, since
the pencil $\lambda \bE-(\bJ-\bR)$ associated with a free \emph{dissipative Hamiltonian} \textsf{DAE} system (i.e.,  $\bu=0$) has many nice properties \cite{MehMW18}: The index of the system is at most two; all eigenvalues are in the closed left half plane; and all eigenvalues on the imaginary axis are semi-simple (except for possibly the eigenvalue $0$, which may have Jordan blocks of size at most two). Furthermore, a singular pencil can only occur when $\bE,\bJ,\bR$ have a common nullspace; see \cite{MehMW19_ppt}. Therefore, if one is able to efficiently compute this common nullspace, it is possible to remove the singular part. 
\begin{lemma}\label{lem:regul} For the \textsf{pHDAE} in
 \eqref{pHDAE}, there exists an orthogonal basis transformation matrix $\bV\in \IR^{n \times n}$
 such that in the new variable $\tbx= \mat{ccc} \tbx_1^{\,T} & \tbx_2^{\,T} & \tbx_3^{\,T}\rix^T = \bV^T \bx$,
 the system has the form
 \begin{eqnarray}\label{structstc}
 \mat{ccc} \bE_1 & 0 & 0 \\
 0 & 0 & 0\\ 0 & 0 & 0 \rix \mat{c} \dot \tbx_1 \\ \dot \tbx_2 \\ \dot \tbx_3 \rix & =&
 \mat{ccc} \bJ_1-\bR_1 & 0 & 0\\ 0 & 0 & 0\\ 0 & 0 & 0\rix\mat{c} \tbx_1 \\ \tbx_2 \\ \tbx_3 \rix
 + \mat{c} \bB_1-\bP_1 \\ \bB_2 \\ 0 \rix \bu,\\
 \by&=&  \mat{ccc} (\bB_1+\bP_1)^T & \bB_2^T  & 0\rix\mat{c} \tbx_1 \\ \tbx_2 \\ \tbx_3 \rix +(\bS+\bN)\bu,
 \nonumber\label{structout}
  \end{eqnarray}
where $\lambda \bE_1-(\bJ_1-\bR_1)$ is regular and $\bB_2$ has full row rank. Also, the subsystem
 \begin{eqnarray}\label{subsystemstate}
 \mat{cc} \bE_1 & 0  \\
 0 & 0   \rix \mat{c} \dot \tbx_1 \\ \dot \tbx_2  \rix & =&
 \mat{cc} \bJ_1-\bR_1 & 0 \\ 0 & 0  \rix\mat{c} \tbx_1 \\ \tbx_2  \rix
 + \mat{c} \bB_1-\bP_1 \\ \bB_2  \rix \bu,\\
 \by&=&  \mat{cc} (\bB_1+\bP_1)^T & \bB_2^T  \rix\mat{c} \tbx_1 \\ \tbx_2  \rix +(\bS+\bN)\bu,
 \nonumber\label{subsystemout}
  \end{eqnarray}
that is obtained by removing the third equation and the variable $\bx_3$ is still a  \textsf{pHDAE}.
  \end{lemma}
\begin{proof}
First determine an orthogonal matrix $\bV_1$ (via SVD) such that
\[
\bV_1^T (\lambda \bE-(\bJ-\bR)) \bV_1=
 \lambda \mat{cc} \bE_1 & 0  \\
 0 & 0 \rix  -
 \mat{cc} \bJ_1-\bR_1 & 0 \\ 0 & 0 \rix,\ \bV_1^T(\bB-\bP)=  \mat{c} \bB_1-\bP_1 \\ \tilde \bB_2-\tilde \bP_2 \rix.
\]
Such $\bV_1$ exists, since  $\bE,\bJ,\bR$ have a common nullspace when the pencil $\lambda \bE-(\bJ-\bR)$ is singular \cite{MehMW19_ppt}.
Then a row compression of $\tilde \bB_2-\tilde \bP_2$ via an orthogonal  matrix $\tilde \bV_2$
and a congruence transformation with  $\bV_2=\textsf{diag}(\bI, \tilde \bV_2)$ is performed, so that with $\bV=
\textsf{diag}(\bV_1, \bV_2)$, we obtain the zero pattern  in \eqref{structstc}. Updating the output equation accordingly and using the fact that the transformed passivity matrix
\[
\tilde \bW=\begin{bmatrix} \bV^T \bR \bV & \bV^T\bP\\ \bP^T \bV^T & \bS\end{bmatrix}\in\mathbb{R}^{(n+m)\times(n+m)}
\]
is still semidefinite; it follows that $\bP_2=0$ and $\bP_3=0$, giving the desired form.
\end{proof}
The  next result presents a condensed form and shows that conditions ({C2}) and ({O2}) are equivalent and hold for a subsystem of (\ref{subsystemstate}). 
 \begin{lemma}\label{lem:c2o2} For the \textsf{pHDAE}  in
 (\ref{subsystemstate}) 
 there exists an orthogonal basis transformation $\hbV$ in the state space and $\bU$ in the control space
 such that in the new variables $\hbx= \mat{ccccc} \hbx_1^{\,T} & \hbx_2^{\,T} & \hbx_3^{\,T}
  \hbx_4^{\,T} & \hbx_5^{\,T} & \hbx_6^{\,T}
 \rix^T = \hbV^T \mat{cc} \tbx_1^{\,T} & \tbx_2^{\,T}
 \rix^T$, and $\mat{ccc
 }  \bu_1^T & \bu_2^T & \bu_3^T \rix^T=  \bU^T\bu$
 the system has the form
  \begin{eqnarray}
&& \mat{cccccc} \bE_{11} &\bE_{12} & 0 & 0 & 0 &0\\
 \bE_{21} & \bE_{22}  & 0 & 0&0& 0\\
 0 & 0 & 0 & 0&0& 0\\ 0 & 0 & 0 &0 &0 & 0\\  0 & 0 & 0 &0 &0 & 0\\  0 & 0 & 0 &0 &0& 0\rix  \mat{c} \dot \hbx_1 \\ \dot \hbx_2 \\ \dot \hbx_3 \\ \dot \hbx_4\\ \dot \hbx_5\rix =
 \mat{cccccc} \bJ_{11}-\bR_{11} &  \bJ_{12}-\bR_{12}& \bJ_{13}-\bR_{13}&\bJ_{14} & \bJ_{15}&0\\
 \bJ_{21}-\bR_{21} &  \bJ_{22}-\bR_{22}& \bJ_{23}-\bR_{23}&\bJ_{24} & 0&0\\
 \bJ_{31}-\bR_{21} & \bJ_{32}-\bR_{32} & \bJ_{33}-\bR_{33} & 0 & 0&0\\
 \bJ_{41} & \bJ_{42} & 0& 0 & 0 & 0\\ \bJ_{51} & 0&0  & 0  & 0 & 0\\ 0& 0 &  0 & 0& 0& 0\rix\mat{c} \hbx_1 \\ \hbx_2 \\ \hbx_3 \\\hbx_4\\ \hbx_5 \\ \hbx_6\rix\nonumber \\
 && \quad \qquad +  \mat{ccc} \bB_{11}-\bP_{11}  & \bB_{12}-\bP_{12} & \bB_{13}-\bP_{13}\\
\bB_{21}-\bP_{21}  & \bB_{22}-\bP_{22} & \bB_{23}-\bP_{23}\\
\bB_{31}-\bP_{31}  & \bB_{32}-\bP_{32} & \bB_{33}-\bP_{33}\\
  0 & \bB_{42} & \bB_{43}\\
   0 & 0 & \bB_{53}\\
 0 & 0 & \bB_{63} \rix \mat{c}  \bu_1 \\ \bu_2 \\\bu_3 \rix,~~~~~~ \label{fullstc}\\
&& \mat{c} \by_1 \\ \by_2 \\ \by_3\rix =
 \mat{cccccc} (\bB_{11}+\bP_{11})^T & (\bB_{21}+\bP_{21})^T & \bB_{31}+\bP_{31})^T  & 0 & 0 & 0 \\
 (\bB_{12}+\bP_{12})^T & (\bB_{22}+\bP_{22})^T & \bB_{32}+\bP_{32})^T  & \bB_{42}^T & 0 & 0 \\
  (\bB_{13}+\bP_{13})^T & (\bB_{23}+\bP_{23})^T & \bB_{33}+\bP_{33})^T  & \bB_{43}^T & \bB_{53}^T & \bB_{63}^T \rix 
  \mat{c} \hbx_1 \\ \hbx_2 \\ \hbx_3 \\ \hbx_4 \\ \hbx_5 \\ \hbx_6\rix
 +\mat{ccc} \bD_{11} & \bD_{12} & \bD_{13}\\
 \bD_{21} & \bD_{22} & \bD_{23}\\
 \bD_{31} & \bD_{32} & \bD_{33}\rix\mat{c}  \bu_1 \\ \bu_2 \\\bu_3 \rix
 ,~~~~
\nonumber \label{fullstcout}
  \end{eqnarray}
  where $\bE_{22}$, $\bJ_{33}-\bR_{33}$, $\bJ_{15}$, and $\bB_{42}$ and $\bB_{63}$  are invertible.
Furthermore, the subsystem obtained by deleting the first, fifth, and sixth block row and column satisfies {\rm (C2)} and equivalently {\rm (O2)}.
\end{lemma}
\begin{proof}
The proof follows again by a sequence of orthogonal transformations.  
Starting from (\ref{subsystemstate}) in the first step 
one determines an orthogonal matrix $\bV_1$ (via a spectral decomposition of $\bE_1$) such
%
$ \tbV_1^T \bE_1 \tbV_1=
\mat{cc} \bE_{11} & 0  \\
 0 & 0 \rix$
 %
with $\bE_{11}>0$, and then forms a congruence transformation with $\hbV_1=\textsf{diag}(\tbV_1,\bI)$ yielding
 \[
 \hbV_1^T(\bJ-\bR) \hbV_1=\mat{cc} \tilde \bJ_{11}-\tilde \bR_{11} & \tilde \bJ_{12}-\tilde \bR_{12} \\ \tilde \bJ_{21}-\tilde \bR_{21}& \tilde \bJ_{22}-\tilde \bR_{22}  \rix,\ \hbV_1^T(\bB-\bP)=  \mat{c} \tilde \bB_1-\tilde \bP_1 \\ \tilde \bB_2-\tilde \bP_2\rix.
\]
Next compute a full rank decomposition 
%
$ \tbV_2^T (\widetilde \bJ_{22}-\widetilde \bR_{22}) \tbV_2=  \mat{cc} \widehat \bJ_{22} -\widehat \bR_{22} & 0 \\ 0& 0 \rix$,
%
where $\widehat \bJ_{22} -\widehat \bR_{22}$ is invertible and  $\widehat \bR_{22}\geq 0$.
This exists, since $\widetilde \bJ_{22}-\widetilde \bR_{22}$ has a positive semidefinite symmetric part.
Then an appropriate  congruence transformation with $\hbV_2=\textsf{diag}(\bI,  \tbV_2, \bI)$
yields
\begin{eqnarray*}
\hbV^T_2\hbV_1^T \bE \hbV_1\hbV_2 &=&
\mat{cccc} \bE_{11} & 0 &0 &0 \\
 0 & 0 & 0 & 0\\ 0 & 0  &0 &0\\ 0 & 0  &0 &0
\rix ,\\
\hbV^T_2 \hbV_1^T(\bJ-\bR) \hbV_1\hbV_2&=&\mat{cccc} \widehat \bJ_{11}-\widehat \bR_{11} & \widehat \bJ_{12}-\widehat \bR_{12} & \widehat \bJ_{13} & 0\\ \widehat \bJ_{21}-\widehat \bR_{21} & \widehat  \bJ_{22}-\widehat \bR_{22} &0 &0 \\ \widehat \bJ_{31} & 0& 0 &0  \\ 0 & 0 & 0 & 0\rix,\
\hbV_2^T\hbV_1^T(\bB-\bP)=  \mat{c} \widehat \bB_1-\widehat \bP_1 \\ \widehat \bB_2-\widehat \bP_2 \\ \widehat \bB_3 \\ \widehat \bB_4\rix,
\end{eqnarray*}
where $\widehat  \bJ_{22}-\widehat \bR_{22}$ is invertible and $\widehat \bB_4$ has full row rank. 

Then one  performs  an orthogonal   decomposition
\[ \tbV_3^T \mat{c} \widehat \bB_3 \\ \widehat \bB_4 \rix \bU= \mat{ccc} 0 & \bB_{42} &\bB_{43}\\ 0 & 0 & \bB_{53}\\
0 & 0 & \bB_{63} \rix
\]
with $\bB_{42}$ and $\bB_{63}$ square  nonsingular, where the number of rows in $\bB_{63}$ is that of $\bB_4$ and applies an appropriate  congruence transformation with $\hbV_3=\textsf{diag}(\bI, \bI,  \hbV_3 , \bI)$  so that one obtains block matrices
  \[
\mat{ccccc} \bE_{1}  & 0 & 0 & 0 &0\\
0 & 0 & 0&0& 0\\
 0 & 0 & 0 & 0 &0 \\ 0 & 0 & 0 &0  & 0\\ 0 & 0 & 0 &0  & 0\rix,
 \mat{ccccc} \bar \bJ_{11}-\bar  \bR_{11} &  \bar \bJ_{12}-\bar \bR_{12}& \bar \bJ_{13}-\bar \bR_{13}&\bar \bJ_{14} &0\\
 \bar \bJ_{21}-\bar \bR_{21} &  \bar \bJ_{22}-\bar \bR_{22}& 0 & 0&0\\
 \bar \bJ_{31}  & 0& 0 & 0 & 0\\  \bar\bJ_{41}&0  & 0  & 0 &0\\ 0 & 0 & 0 & 0 & 0\rix,
\]
As final step one computes a column compression of the full row rank matrix  $\bar \bJ_{41}$ and applies a an appropriate congruence transformation. This yields the desired form.

The fact that the several $\bP$ blocks  do not occur follows again from the semidefiniteness of the  transformed passivity matrix.
Since $\bJ_{51}$, $\bB_{42}$ and $\bB_{63}$ are invertible, it follows immediately that $\bu_3=0$ and $\hat \bx_1=0$ and that $\hat \bx_5$  is uniquely determined by all the other variables. Considering the subsystem obtained by removing the first, fifth, and sixth block row and column, it follows by the  symmetry structure  that the condition ({C2}) holds  if and only  ({O2}) holds.
\end{proof}

The  procedure to compute the condensed form (\ref{structstc}) immediately separates the 
dynamical part (given by the second block row), the  algebraic index-$1$ condition (given by the  third block rows), and the index-$2$ conditions (given by the fifth and first block row).  
In the following, for ease of notation,  we 
denote the system that is obtained by removing the variables $\widehat \bx_1$, $\widehat \bx_4$
and the corresponding first and fifth equation by
\[
\begin{aligned}
\bE_r\dot{\bx}_r&=(\bJ_r-\bR_r)\, \bx_r+(\bB_r-\bP_r)\,\bu,\ \bx_r(t_0)=0,\\
\by_r&=(\bB_r+\bP_r)^T\,\bx+(\bS_r+\bN_r)\,\bu.
\end{aligned}
\]
If we apply an output feedback
$\bu=-\bK_r\by_r$ with $\bK_r=\bK_r^T> 0 $, that makes the (2,2) block in the closed loop system invertible, then
this corresponds to a feedback
for the state-to-output map in  (\ref{subsystemstate}) of the form
$\by_r=(\bI  +(\bS_r+\bN_r)\bK_r)^{-1}((\bB_r+\bP_r)^T\bx_r$, which we can insert into the
first equation to obtain
\begin{eqnarray*}
\bE_r\dot{\bx_r}&=&(\bJ_r-\bR_r) \bx_r+(\bB_r-\bP_r) \bK_r\by_r\\
&=&\left ((\bJ_r-\bR_r)-(\bB_r-\bP_r) (\bK_r^{-1} +(\bS_r+\bN_r))^{-1}((\bB_r+\bP)^T\right )\bx_r.
\end{eqnarray*}
The negative of the symmetric part of the system matrix is the Schur complement of 
\[
\tilde \bW_r=\mat{cc}  \bR_r &\bP_r \\ \bP_r^T  & \bK_r^{-1}+\bS_r+\bN_r\rix\geq 0.
\]
Hence the closed loop system is regular and the closed loop system is still \textsf{pH}.

The procedures described above are computationally demanding, since they typically require large-scale singular value decompositions or spectral decompositions. Fortunately, in many practical cases the condensed form is already available  directly from the modeling procedure, so that the transfer function can be formed and the model reduction method can be directly applied.

 \section{Interpolatory model reduction of \textsf{pHDAE}s}  \label{sec:intmodred}
Given an order-$n$  \textsf{pHDAE} as in \eqref{pHDAE}, we want to construct an order-$r$ reduced \textsf{pHDAE},
with $r \ll n$,  having the same structured form
\begin{equation}\label{pHDAEred}
\begin{aligned}
\Er\dot{\bx}_r&=(\Jr-\Rr)\, \bx_r+(\Br-\Pred)\,\bu,\ \bx_r(t_0)=0,\\
\by_r&=(\Br+\Pred)^T\,\bx_r+(\Sr+\Nr)\,\bu,
\end{aligned}
\end{equation}
such that $\Er$, $\Jr$, $\Rr\in\IR^{r \times r}$, $\Br, \Pred\in\IR^{r \times m}$, $\Sr=\Sr^T$, $\Nr=-\Nr^T\in\IR^{m \times m}$ satisfy the same requirements as in Definition \ref{pHDAEDEF} and that the output $\by_r(t)$ of \eqref{pHDAEred} is an accurate approximation to the original output $\by(t)$ over a wide range of admissible inputs $\bu(t)$. We will enforce accuracy by constructing  the reduced model \eqref{pHDAEred} via interpolation.

Let $\bH(s) = \cbC(s\bE - \bA)^{-1} \cbB+\bD$ and $\Hr(s) = \cbCr(s\Er - \Ar)^{-1} \cbBr+\Dr$ denote the transfer functions of \eqref{pHDAE} and \eqref{pHDAEred}, where $\cbC = (\bB+\bP)^T$,  $\bA = \bJ-\bR$, $\cbB = \bB - \bP$, $\bD = \bS+\bN$, and similarly for the reduced-order (``hat'') quantities. Given
the right-interpolation points
$\{\sigma_1,\sigma_2,\ldots,\sigma_r\} \in \IC$ together with the corresponding left-tangent directions
$\{\Rdir_1,\Rdir_2,\ldots,\Rdir_r\} \in \IC^m$ and
the left-interpolation points
$\{\mu_1,\mu_2,\ldots,\mu_r\} \in \IC$ together with the corresponding right-tangent directions
$\{\Ldir_1,\Ldir_2,\ldots,\Ldir_r\} \in \IC^m$,
 we would like to construct $\Hr(s)$ such that it tangentially interpolates  $\bH(s)$, i.e.,
 \begin{equation}
 \bH(\sigma_i) \Rdir_i  = \Hr(\sigma_i) \Rdir_i~~~\mbox{and}~~~
  \Ldir_i^T\bH(\mu_i)   = \Ldir_i^T\Hr(\mu_i),~~\mbox{for}~~i=1,2,\ldots,r.
\end{equation}
These tangential interpolation conditions can be easily enforced via a Petrov-Galerkin projection  \cite{gallivan2005model,BG2017siam,AntBG20}.
Construct $\bV\in \IC^{n \times r} $ and $\bZ \in \IC^{n \times r}$ using
\begin{align}  \label{eqV}
\bV & =  \left[(\sigma_1 \bE - \bA)^{-1}\cbB \Rdir_1,~~(\sigma_2 \bE - \bA)^{-1}\cbB \Rdir_2,~~\cdots~(\sigma_r \bE - \bA)^{-1}\cbB \Rdir_r\right] ~~\mbox{and}~~\\
\bZ  &=  \left[(\mu_1 \bE - \bA)^{-T}\cbC^T \Ldir_1,~~(\mu_2 \bE - \bA)^{-T}\cbC^T \Ldir_2,~~\cdots~(\mu_r \bE - \bA)^{-1}\cbC^T \Ldir_r\right].\label{eqW}
\end{align}
Then the interpolatory reduced  model can be obtained  via projection:
\begin{equation} \label{pgred}
\Er = \bZ^T \bE \bV,~~\Ar = \bZ^T \bA \bV,~~\cbBr = \bZ^T \cbB,~~\cbCr =\cbC \bV,~~\mbox{and}~~\Dr = \bD.
\end{equation}
In the setting of \textsf{pHDAE}s two fundamental issues arise: First, the reduced quantities in \eqref{pgred} are no longer guaranteed to have the \textsf{pH} structure. This is easiest to see in
the reduced quantity $\cbAr =  \bZ^T \cbA \bV =  \bZ^T \bJ \bV - \bZ^T \bR \bV$. If we decompose $\cbAr$ into its symmetric and skew-symmetric parts, we can no longer guarantee that the symmetric part is positive semi-definite. This could be resolved by using a Galerkin projection, i.e., with $\bZ = \bV$.
In this case one only satisfies  the interpolation conditions associated with right interpolation data. However, this does not resolve the second issue since in the generic case when $r < \rank(\bE)$, the reduced quantity
$\Er$ is expected to be a nonsingular matrix; thus the reduced system will be an \textsf{ODE}. This means that the polynomial parts of $\bH(s)$ and $\Hr(s)$ do not match, leading to unbounded errors.

Structure-preservation interpolatory reduction of   \textsf{pH}   systems in the most general setting of tangential interpolation has been studied in
\cite{GugPBS12,GugPBS09}. However, this work focused on the \textsf{ODE} case. On the other hand,
\cite{GugSW13} developed the tangential interpolation framework for reducing \emph{unstructured} \textsf{DAE}s with guaranteed polynomial matching. Only recently in \cite{EggKLMM18,HauMM19_ppt}, the combined problem has been investigated.
We now develop a treatment of structure-preserving interpolatory model reduction problem
for index-$1$ and index-$2$ \textsf{pHDAE}s in the  general setting of tangential interpolation.

\subsection{Semi-explicit index-$1$ \textsf{pHDAE} systems}\label{sec:indexone}
The simplest class of \textsf{pHDAE}s are \emph{semi-explicit index-$1$} \textsf{pHDAE}s
of the form
\begin{align}
\begin{array}{rcl} \left[
\begin{array} {cc}
\bE_{11} & 0 \\ 0 & 0
\end{array} \right] \dot{\bx}(t)& = &
\left[\begin{array} {rr}
\bJ_{11}- \bR_{11} & \bJ_{12}- \bR_{12} \\ -\bJ_{12}^T- \bR_{12}^T & \bJ_{22}- \bR_{22}
\end{array} \right]  \bx(t) + \left[ \begin{array} {c}
\bB_{1}-\bP_1  \\ \bB_2-\bP_2
\end{array} \right] u(t), \\[3ex] \by(t) &= & \left[ \begin{array} {cc} \bB_1^T+\bP_1^T& \bB_2^T+\bP_2^T \end{array} \right] \bx(t) + (\bS+\bN) \bu(t).
\end{array}
 \label{daeind1B2nonzero}
\end{align}
where $\bE_{11}$ and $\bJ_{22}-\bR_{22}$ are nonsingular.  We 
have the following interpolation result.
\begin{theorem}
    \label{ind1B2nonzerosol1}
Consider the \textsf{pHDAE}  system  in \eqref{daeind1B2nonzero}.
 Let the interpolation points
$\{\sigma_1,\sigma_2,\ldots,\sigma_r\} \in \IC$ and the corresponding tangent directions
$\{\Rdir_1,\Rdir_2,\ldots,\Rdir_r\} \in \IC^m$ be given.
Construct the interpolatory model reduction basis $\bV$ as
\begin{equation} \label{Vind1}
\bV = \begin{bmatrix} \bV_1 \\ \bV_2 \end{bmatrix} =  \left[(\sigma_1 \bE - \bA)^{-1}\cbB \Rdir_1,
~~\cdots,~~(\sigma_r \bE - \bA)^{-1}\cbB \Rdir_r\right] \in \IC^{n \times r},
\end{equation}
where $\bV$ is partitioned conformably with the system,
and define the matrices
\begin{equation*}
\bbB = \begin{bmatrix}\Rdir_1 & \Rdir_2 & \cdots & \Rdir_r \end{bmatrix}\in \IC^{m \times r} ~ \mbox{and} ~
\cbD = \bD - (\bB_2^T+\bP_2^T) (\bJ_{22}-\bR_{22})^{-1} (\bB_2-\bP_2) \in \IC^{m \times m}.
\end{equation*}
Let
$
\bA=\mat{cc} \bA_{11} & \bA_{12}\\ \bA_{21} & \bA_{22} \rix=\bJ-\bR,
$
partitioned accordingly to \eqref{daeind1B2nonzero}. Then the transfer function $\Hr(s)$ of the reduced model
\begin{align} \label{romdaeind1B2nonzero}
\begin{array}{rclcrcl}
\Er \dot{\bx}_r(t)& =& (\Jr-\Rr) \bx_r(t) + \cbBr \bu(t),  &~~~~&  \by_r(t) &= &\cbCr \bx_r(t) + \Dr \bu(t)
\end{array}
\end{align}
with
\begin{align} \label{eqn:case2}
\begin{array}{rclcrcl}
\Er &= & \bV_1^T\bE_{11}\bV_1
& ~~~ & \cbCr &=& \cbC \bV + (\bB_2^T+\bP_2^T) \bA_{22}^{-1} (\bB_2-\bP_2)\bbB , \\
\Ar &= & \bV^T \cbA \bV + \bbB^T (\cbD-\bD) \bbB, & ~~~ &  \Dr &=& \cbD = \bD - (\bB_2^T+\bP_2^T) \bA_{22}^{-1} (\bB_2-\bP_2)\nonumber \\
\cbBr&=& \bV^T \cbB + \bbB^T(\bB_2^T+\bP_2^T) \bA_{22}^{-1} (\bB_2-\bP_2),
\end{array}
\end{align}
matches the polynomial part of $\bH(s)$  and tangentially interpolates it, i.e.,
\[
\bH(\sigma_i) \Rdir_i = \Hr(\sigma_i) \Rdir_i,~~\mbox{for}~~i=1,2,\ldots,r.
\]
Define $ \Pred=\frac 12 (-\cbGr+ \cbCr)$, and decompose
$\Dr= \Sr + \Nr$, $\Ar=  \Jr -\Rr$ into their symmetric and skew-symmtric part.
%
%
Then, the reduced model \eqref{romdaeind1B2nonzero} is a  \textsf{pHDAE}  system if
the reduced passivity matrix
$
\widehat \bW =\mat{cc} \Rr &  \Pred \\ \Pred^T &  \Sr\rix
$
is positive semidefinite.
\end{theorem}

\begin{proof}
We employ a Galerkin projection using the interpolatory model reduction basis $\bV$ to obtain the \emph{intermediate} reduced model
\begin{align*}
\tEr = \bV_1^T\bE_{11}\bV_1,~~
\tAr = \tJr - \tRr  = \bV^T \bJ \bV - \bV^T \bR \bV,~~
\tcbB= \bV^T \cbB, ~~\tcbC= \cbC \bV~~\mbox{and}~~\tDr = \bD.
\end{align*}
Even though this reduced model is a \textsf{pHDAE}  system due to the one-sided Galerkin model reduction and
it satisfies the tangential interpolation conditions,
it  will \emph{not} match the transfer $\bH(s)$ function at $s=\infty$, i..e, its polynomial part, given by
\[
\lim_{s\to \infty} \bH(s) =
\cbD = \bD - (\bB_2^T +\bP_2^T)\bA_{22}^{-1} (\bB_2-\bP_2).
\]
A remedy to this problem,  proposed in \cite{MayA07,beattie2008ipm} and employed in the general \textsf{DAE} setting in \cite{GugSW13}, is to modify  the $\bD$-term in the reduced model  to match the polynomial part and at the same time  \emph{to shift} the other reduced  quantities appropriately so as to keep the tangential interpolation property.  Using this, we obtain a modified  reduced model
\begin{align}
\Er &= \tEr  = \bV_1^T\bE_{11}\bV_1,~~ \nonumber \\
\Ar &= \tAr + \bbB^T (\cbD-\bD) \bbB   =  \bV^T \bJ \bV - \bV^T \bR \bV - \bbB^T(\bB_2^T +\bP_2^T)A_{22}^{-1} (\bB_2-\bP_2)\bbB ,\nonumber \\
\cbBr&= \tcbB + \bbB^T(\bB_2^T +\bP_2^T)A_{22}^{-1} (\bB_2-\bP_2) \nonumber\\
&= \bV_1^T (\bB_1 -\bP_1)+ \bV_2^T
(\bB_2 -\bP_2)+  \bbB^T(\bB_2^T +\bP_2^T)A_{22}^{-1} (\bB_2-\bP_2),\\
\cbCr &= \tcbC + (\bB_2^T +\bP_2^T)A_{22}^{-1} (\bB_2-\bP_2) \bbB\nonumber            \\
&= (\bB_1^T+\bP_1^T) \bV_1^T + (\bB_2^T +\bP_2^T)\bV_2^T +(\bB_2^T +\bP_2^T)A_{22}^{-1} (\bB_2-\bP_2) \bbB ,~~~~\mbox{and}\nonumber\\
\Dr &= \cbD = \bD - (\bB_2^T +\bP_2^T)A_{22}^{-1} (\bB_2-\bP_2),\nonumber
\end{align}
which  satisfies the original tangential interpolation conditions, and  matches the polynomial part due to the modified $\Dr$-term. However, for this  system to be \textsf{pH}
we need to check that the associated passivity matrix is still positive semidefinite, which after rewriting the input and output matrix in the usual way is exactly  the condition on $\widehat \bW$ in the assertion.
We then have that the reduced model not only  satisfies the interpolation conditions and matches the polynomial part at $s=\infty$, but also is \textsf{pH}.
\end{proof}

\begin{remark}\label{rem:no input}{\rm
Note that if the input does not influence the algebraic equations, i.e., if $\bB_2 = \bP_2=0$, then the shift of the constant term is not necessary, and the formulas simplify significantly, i.e., $\Ar =  \bV^T \cbA \bV$, $\cbBr= \tcbB = \bV^T \cbB$, $\cbCr = \cbC \bV$, and $\Dr = \cbD = \bD$.}
 \end{remark}

Another solution to preserving  \textsf{pH}  structure via interpolation can be obtained  through the following theorem.
 \begin{theorem}  \label{ind1B2nonzerosol2}
Consider a full-order  \textsf{pHDAE}  system of the form \eqref{daeind1B2nonzero}.
 Let interpolation points
$\{\sigma_1,\sigma_2,\ldots,\sigma_r\} \in \IC$ and the corresponding tangent directions
$\{\Rdir_1,\Rdir_2,\ldots,\Rdir_r\} \in \IC^m$ be given.
Construct the interpolatory model reduction basis $\bV$ as in \eqref{Vind1}.  Then the reduced model

\vspace{-2ex}
{\small
\begin{align}
\begin{array}{rcl}
 \left[
\begin{array} {cc}
\widehat{\bV}_1^T\bE_{11} \widehat{\bV}_1 & 0 \\ 0 & 0
\end{array} \right] \dot{\bx}(t)& = &
\left[\begin{array} {rr}
{\bV}_1^T(\bJ_{11}- \bR_{11}){\bV}_1 & {\bV}_1^T(\bJ_{12}- \bR_{12}) \\ (-\bJ_{12}^T- \bR_{12}^T){\bV}_1 & \bJ_{22}- \bR_{22}
\end{array} \right]  \bx(t) + \left[ \begin{array} {c}
{\bV}_1^T(\bB_{1}-\bP_1)  \\ \bB_2-\bP_2
\end{array} \right] \bu(t) \\[3ex] \by_r(t) &= &  \left[ \begin{array} {cc} (\bB_1^T+\bP_1)^T{\bV}_1& \bB_2+\bP_2^T \end{array} \right] \bx(t) + \bD \bu(t)
\end{array}
~ \label{romdaeind1B2nonzerosol2}
\end{align}
}

\vspace{-2ex}
\noindent
retains the  \textsf{pH}  structure, tangentially interpolates the original model, and matches the polynomial part.
\end{theorem}
 \begin{proof}
We first note that the subspace spanned by the columns of $\bV = \begin{bmatrix} \bV_1^T & \bV_2^T \end{bmatrix}^T $ is contained in the subspace spanned by the columns of
$
 \widehat{\bV} := \textsf{diag}(\bV_1,\bI). 
$
Then the system in \eqref{romdaeind1B2nonzerosol2} results from reducing the original system in \eqref{daeind1B2nonzero} via  $\widehat{\bV}$.
Since $\mathsf{span}(\bV) \subseteq \mathsf{span}(\widehat{\bV})$, this reduced \textsf{DAE}  automatically satisfies the interpolation conditions and since $ \widehat{\bV}$ does not alter the matrix $\bJ_{22} -\bR_{22}$ and the matrices $\bB_2,\bP_2$, the  polynomial part of the transfer function  is
$ \bD - (\bB_2^T+\bP_2^T) (\bJ_{22} -\bR_{22})^{-1} (\bB_2-\bP_2)$, matching that of the original model.
The reduced  system in \eqref{romdaeind1B2nonzerosol2} is  \textsf{pH}  as this one-sided projection retains the original  \textsf{pH}  structure.   \end{proof}

\begin{remark}\label{rem:order}{\rm
 Theorem \ref{ind1B2nonzerosol2} presents a seemingly easy solution to structure-preserving interpolatory model reduction of index-1 \textsf{pHDAE}s compared to Theorem \ref{ind1B2nonzerosol1}.
However,  this 
may not be the maximal reduction that is possible because redundant algebraic conditions cannot be removed; see \cite{MehS05}.
 }
\end{remark}
\subsection{Semi-explicit  \textsf{pHDAE}  systems with index-$2$ constraints}
\label{sec:index2cases}
As next class we study semi-explicit index-$2$
systems. We first consider the case that the input does not affect the algebraic equations.
\begin{theorem}\label{thm:indtwo1}
Consider an index-$2$  \textsf{pHDAE}  system of the form
\begin{align}
\begin{array}{rcl}
\left[
\begin{array} {cc}
\bE_{11} & 0 \\ 0 & 0
\end{array} \right]
\left[ \begin{array}{c}
\dot{\bx}_1(t) \\
\dot{\bx}_2(t)
\end{array}
\right]
& = &
\left[\begin{array} {cc}
\bJ_{11}- \bR_{11} & \bJ_{12} \\ -\bJ_{12}^T & \mathbf{0}
\end{array} \right] \left[ \begin{array}{c}
{\bx}_1(t) \\
{\bx}_2(t)
\end{array}
\right] + \left[ \begin{array} {c}
\bB_{1} -\bP_1 \\ \mathbf{0}
\end{array} \right] \bu(t)   \\[3ex]  \by(t) &= & \left[ \begin{array} {cc} \bB_1^T + \bP_1^T& \mathbf{0}  \end{array} \right] \left[ \begin{array}{c}
{\bx}_1(t) \\
{\bx}_2(t)
\end{array}
\right]  + \bD \bu(t),
\end{array}
~\label{eqn:ind2case1}
\end{align}
with $\bE_{11}>0$ and set $\bA_{11}=\bJ_{11}- \bR_{11}$.  Given interpolation points $\{\sigma_1,\sigma_2,\ldots,\sigma_r\}$ and associated tangent directions $\{ \Rdir_1, \Rdir_2,\dots, \Rdir_r\}$, let the vectors $\bv_i$, for $i=1,2,\ldots,r$, be the first block of the solution of
\begin{equation} \label{eqn:Vind2case}
\left[ \begin{array}{cc} \bA_{11}-\sigma_i\bE_{11}  & \bJ_{12}\\  -\bJ_{12}^T & \mathbf{0} \end{array} \right]
\left[ \begin{array}{cc} \bv_i \\  \bz \end{array} \right]=
\left[ \begin{array}{cc} (\bB_1-\bP_1)\Rdir_i \\  \mathbf{0} \end{array} \right].
\end{equation}
Define $\bV = \left[\bv_1, \bv_2, \dots, \bv_r\right]$.
Then the reduced model
\begin{equation}  \label{eqn:romind2case}
\Er \dot\bx_r = (\Jr - \Rr) \bx_r + \cbBr \bu(t),~~\by_r = \cbCr \bx_r + \Dr,
\end{equation}
\begin{align}
\mbox{with}\hspace{9ex}\Er &=  \bV^T\bE_{11}\bV,~~ \nonumber
\Jr  =  \bV^T \bJ_{11} \bV,~~ \Rr =  \bV^T \bR_{11} \bV, \\
\cbBr&=  \bV^T \bB_1 - \bV^T \bP_1,~~\cbCr= \bB_1^T \bV^T + \bP^T \bV_1^T,~~~~\mbox{and}~~~\Dr = \bD,
\hspace{4ex}
\label{redssind2case1}
\end{align}
is still \textsf{pH},  matches the polynomial part of the original transfer function, and satisfies the tangential interpolation conditions, i.e.,
\[
\bH(\sigma_i) \Rdir_i = \Hr(\sigma_i) \Rdir_i,~~\mbox{for}~~i=1,2,\ldots,r.
\]
\end{theorem}
\begin{proof}
Note first that the regularity of $\lambda\bE-(\bJ-\bR)$ and the index-$2$ condition imply that $-\bJ_{12}^T\bE_{11}^{-1}\bJ_{12}$ is invertible, see \cite{KunM06,BeaMXZ18}.
Following \cite{GugSW13}, we write  \eqref{eqn:ind2case1} as
\begin{equation}   \label{37sys}
\arraycolsep=2pt
\begin{array}{rcl}
\bfpi\bE_{11}\bfpi \dot{\bx}_1(t) & = & \bfpi\bA_{11}\bfpi \bx_1(t) + \bfpi\bB_1 \bu(t), \\
 \by(t) &=& \bC_1\bfpi\bx_1(t) +\bD\bu(t),
 \end{array}
 \end{equation}
together with the algebraic equation
\[
\bx_2(t)= -(\bJ_{12}^T\bE_{11}^{-1}\bJ_{12})^{-1}\bJ_{12}^T\bE_{11}^{-1}\bA_{11}\bx_1(t)
-(\bJ_{12}^T\bE_{11}^{-1}\bJ_{12})^{-1}\bJ_{12}^T\bE_{11}^{-1}\bB_1 \bu(t),
\]
with the projector 
\[
\bfpi_  =  \bI-\bE_{11}^{-1}\bJ_{12}(\bJ_{12}^T\bE_{11}^{-1}\bJ_{12})^{-1}\bJ_{21}^T
\]
The equivalent system \eqref{37sys} is now an implicit \textsf{ODE}  \textsf{pH}  system that can be reduced with standard model reduction techniques. However, this would require computing the projector $\bfpi$explicitly, see \cite{MehS05}. For general index-$2$ \textsf{DAE} systems  one can avoid  this computational step  in interpolatory model reduction,
see \cite{GugSW13}.

To adapt this idea to  \textsf{pHDAE}  systems, we construct $\bV$ using \eqref{eqn:Vind2case} and then compute the reduced-order quantities via one-sided projection as in \eqref{redssind2case1}.  This construction of $\bV$, as in \cite{GugSW13},
guarantees that the reduced model in \eqref{eqn:romind2case}  tangentially interpolates the original  \textsf{pHDAE}  system in \eqref{37sys} and the polynomial part of the transfer function in \eqref{eqn:ind2case1} is given by $\bD=\bS+\bN$ partitioned in its symmetric and skew-symmetric part.
Since \eqref{eqn:romind2case} is an implicit \textsf{ODE} \textsf{pH}  system  with the exact $\bD$-term, it matches the polynomial part of the original transfer function $\bH(s)$.

It remains to show that \eqref{eqn:romind2case} is \textsf{pH}.
By construction in  \eqref{redssind2case1},  $\Jr$ is skew-symmetric, $\Rr$ is symmetric positive semidefinite, and $\bE_{11}$ is symmetric positive definite.  Moreover,
\[
\left[ \begin{array}{cc}
\bV^T \bR_{11} \bV & \bV^T \bP_1 \\
 \bP_1^T \bV & \bS  \end{array}
\right] =
\left[ \begin{array}{cc}
\bV^T  & 0 \\
0  & \bI \end{array}
\right]
\left[ \begin{array}{cc}
 \bR_{11}  &  \bP_1 \\
 \bP_1^T & \bS  \end{array}
\right]
\left[ \begin{array}{cc}
\bV  & 0 \\
0  & \bI \end{array}
\right]
 \geq 0,
\]
since the original model is PH, and therefore the pH-structure is retained.
\end{proof}

The situation 
becomes more complicated when the second block in $\cbB$ is nonzero, i.e., the system has the form
\begin{align}
\begin{array}{rcl}
\left[
\begin{array} {cc}
\bE_{11} & 0 \\ 0 & 0
\end{array} \right]
\left[ \begin{array}{c}
\dot{\bx}_1(t) \\
\dot{\bx}_2(t)
\end{array}
\right]
& =&
\left[\begin{array} {cc}
\bJ_{11}- \bR_{11} & \bJ_{12} \\ -\bJ_{12}^T & \mathbf{0}
\end{array} \right] \left[ \begin{array}{c}
{\bx}_1(t) \\
{\bx}_2(t)
\end{array}
\right] + \left[ \begin{array} {c}
\bB_{1}-\bP_1  \\ \bB_2-\bP_2
\end{array} \right] \bu(t),  \\[3ex] \by(t) &= & \left[ \begin{array} {cc} \bB_1^T+\bP_1^T&  \bB_2+\bP_2^T  \end{array} \right] \left[ \begin{array}{c}
{\bx}_1(t) \\
{\bx}_2(t)
\end{array}
\right]  + \bD \bu(t).
\end{array}
~\label{eqn:ind2case21}
\end{align}
We have the following theorem.
\begin{theorem}
Consider a  \textsf{pHDAE}  system of the form \eqref{eqn:ind2case21} and define the matrices
 \begin{align*}
\bZ & = -(\bA_{21}\bE_{11}^{-1}\bA_{12})^{-1} = \left(\bJ_{12}^T \bE_{11}^{-1}\bJ_{12}\right)^{-1}  \\
 \cbC & = 
  (\bB_1^T-\bP_1^T) - (\bB_2^T-\bP_2^T)\bZ \bJ_{12}^T\bE_{11}^{-1}(\bJ_{11}- \bR_{11}), \\
 \cbB &= 
 (\bB_1 -\bP_1) + (\bJ_{11}- \bR_{11})\bE_{11}^{-1}\bJ_{12} \bZ(\bB_2-\bP_2), \\
 \cbD_0 &= 
 \bD - (\bB_2^T+\bP_2^T)\bZ\bJ_{12}^T\bE_{11}^{-1}(\bB_1-\bP_1), \quad \mbox{and} \label{Dtilde} \\
  \cbD_1 &= 
  -(\bB_2^T+\bP_2^T) \bZ (\bB_2-\bP_2).
  \end{align*}
 Given interpolation points $\{\sigma_1,\sigma_2,\ldots,\sigma_r\}$ and associated  tangent directions $\{ \Rdir_1, \Rdir_2,\dots, \Rdir_r\}$, let  the vectors $\bv_i$ be the first blocks of the solutions of
\begin{equation} \label{eqn:Vind2case2}
\left[ \begin{array}{cc}\bJ_{11}- \bR_{11}-\sigma_i\bE_{11} & \bJ_{12}\\  -\bJ_{12}^T & \mathbf{0} \end{array} \right]
\left[ \begin{array}{cc} \bv_i \\  \bz \end{array} \right]=  \bB\Rdir_i,~~\mbox{for}~~i=1,2,\ldots,r,
\end{equation}
and  set $\bV := \left[\bv_1, \bv_2, \dots, \bv_r\right]$, $\bu_1:=\bu$, $\bu_2:=\dot \bu$, and
$\Dr:=\mat{cc} \cbD_0 & \cbD_1 \rix=\Sr+ \Nr$.
Then, the reduced model
\begin{equation}  \label{eqn:romind2case2}
\Er \dot\bx_r = (\Jr - \Rr) \bx_r + \mat{cc}\Br-\Pred & 0\rix \mat{c}\bu_1\\ \bu_2\rix,~~\by_r = (\Br+\Pred)^T \bx_r + \Dr \mat{c}\bu_1\\ \bu_2\rix,
\end{equation}
\begin{align}
\mbox{with}\hspace{10ex}\Er &=  \bV^T\bE_{11}\bV,~~  
\Jr  =  \bV^T \bJ_{11} \bV,~~  \Rr = \bV^T \bR_{11} \bV,\nonumber\\
\Br&=  \frac{1}{2}\left(\bV^T \cbB + \bV^T \cbC^T\right),~~\mbox{and}~~ \bPr = \frac{1}{2}\left(\bV^T \cbC^T-\bV^T \cbB\right),
\hspace{0.3in}
\label{redB}
\end{align}
satisfies the interpolation conditions, matches the polynomial part of the transfer function, and preserves the  \textsf{pH}  structure, provided that the reduced passivity matrix
$\widehat \bW=\mat{cc} \Rr & \Pred \\ \Pred^T &  \Sr\rix$
is positive semidefinite.
\end{theorem}
\begin{proof}
The proof follows similar to the proof of Theorem~\ref{thm:indtwo1}.
Following \cite{GugSW13},
the state $\bx_1$ can be
decomposed as \mbox{$\bx_1 = \bx_c + \bx_g$}, where
$\bx_g = \bE_{11}^{-1}\bA_{12}(\bJ_{12}^T\bE_{11}^{-1}\bJ_{12})^{-1}(\bB_2-\bP_2)\bu(t)$
and $\bx_c(t)$ satisfies $\bJ_{12}^T\bx_c = 0$. Then,
one can rewrite  \eqref{eqn:ind2case21} as
\begin{equation}   \label{newsystem}
\arraycolsep=2pt
\begin{array}{rcl}
\bfpi\bE_{11}\bfpi \dot{\bx}_0(t) & = & \bfpi\bA_{11}\bfpi \bx_0(t) + \bfpi\bB_1 \bu(t) \\
 \by(t) &=& \bC \bfpi\bx_1(t) +\cbD_0\bu(t) +  \cbD_1\dot{\bu}(t).
 \end{array}
 \end{equation}
As before, the \textsf{ODE} part can be reduced with usual model reduction techniques.
Following  \cite{GugSW13}, however, we  achieve this without computing the projectors $\bfpi_l$ and $\bfpi_l$ explicitly, instead
    by constructing $\bV$ using \eqref{eqn:Vind2case2} and then applying one-sided model reduction with $\bV$ to obtain the reduced model
\begin{equation}  \label{introm}
\Er \dot\bx_r = (\Jr - \Rr) \bx_r + \tBr \bu(t),~~\by_r = \tCr^T \bx_r + \cbD_0 \bu(t) +
\cbD_1 \dot \bu(t),
\end{equation}
where
$\Er =  \bV^T\bE_{11}\bV$,
$\Jr =  \bV^T \bJ_{11} \bV$, $\Rr =  \bV^T \bR_{11} \bV$,
$\tBr= \bV^T \cbB$,  and $\tCr=  \cbC \bV$.
This reduced model, by construction, satisfies the tangential interpolation conditions.
Note that the reduced model in \eqref{introm} has exactly the same realization as the reduced model in \eqref{eqn:romind2case2}  except for  the reduced $\tBr$ and $\tCr$ terms.
The reduced terms in \eqref{introm} already have the desired  \textsf{pH}  structure.  To recover
the  symmetry in  \eqref{newsystem}, we deternmine matrices $\Br$ and $\bPr$ such that
$\tBr = \bV^T \cbB = \Br - \bPr$ and $\tCr = \cbC \bV = (\Br + \bPr)^T$ via
%
%
$$ \Br =  \frac{1}{2}\left(\tBr + \tCr^T\right) =  \frac{1}{2} \bV^T \left(\cbB +\cbC^T\right) \quad \mbox{and} \quad
\bPr =  \frac{1}{2}\left(\tCr^T - \tBr\right) =  \frac{1}{2} \bV^T \left(\cbC^T -  \cbB \right),$$
recovering \eqref{redB}.
The final requirement to retain the  \textsf{pH}  structure is, then, again that
$\left[
\begin{array}{cc}
\Rr & \bPr \\
\bPr^T & \frac{1}{2}\left( \bD + \bD^T \right)
\end{array}
\right] \geq 0$,
which is the final condition in the statement of the theorem.
\end{proof}

This approach has the disadvantage that one has to introduce the derivative of $\bu$ as an extra input, which may lead to difficulties  when applying standard control and optimization methods. For this reason it is usually preferable to first perform an index reduction 
via an appropriate output feedback, see Section~\ref{sec:regularization}, 
and then apply the results from Section~\ref{sec:indexone}.
But note that this changes the polynomial part of the transfer function.
\subsection{Semi-explicit  \textsf{pHDAE}  systems with index-$1$ and index-$2$ constraints}
\label{sec:index12}
Finally we consider semi-explicit index-$2$
systems which also have an index-$1$ part, \cite{
EggKLMM18,
HauMM19_ppt,MehS05}. In this case we only consider the special case with state
\begin{equation}
 \mat{ccc} \bE_{11} & \bE_{22} & 0 \\
 \bE_{21} & \bE_{22} & 0 \\ 0 & 0 & 0 \rix \mat{c} \dot \bx_1 \\ \dot \bx_2 \\ \dot \bx_3 \rix  =
 \mat{ccc} \bJ_{11}-\bR_{11} &  \bJ_{12}-\bR_{12}& \bJ_{13} \\
 \bJ_{21}-\bR_{21} & \bJ_{22}-\bR_{22} & 0 \\ \bJ_{31} & 0 & 0  \rix\mat{c} \bx_1 \\ \bx_2 \\ \bx_3\rix
 \label{index12state}
 + \mat{c} \bB_1-\bP_1 \\ \bB_2-\bP_2 \\ 0\rix \bu,
 \end{equation}
 output $
 \by=  \mat{ccc} (\bB_1+\bP_1)^T & (\bB_2+\bP_2)^T  & 0\rix\mat{c} \bx_1 \\ \bx_2 \\ \bx_3 \rix
  +(\bS+\bN)\bu$,
 $ \mat{cc} \bE_{11} & \bE_{22} \\
 \bE_{21} & \bE_{22}\rix>0$, $\bJ_{22} -\bR_{22}$, and $\bJ_{31}$ are nonsingular.
\begin{theorem}\label{thm:indtwoindone}
Consider an index-$2$  \textsf{pHDAE}  system of the form (\ref{index12state}). 
Construct an interpolatory model reduction basis
\begin{equation} \label{Vind1ind2}
\bV = \begin{bmatrix} \bV_1^T & \bV_2^T & \bV_3^T \end{bmatrix}^T =  \left[(\sigma_1 \bE - \bA)^{-1}\cbB \Rdir_1,~~\cdots~(\sigma_r \bE - \bA)^{-1}\cbB \Rdir_r\right] \in \IC^{n \times r},
\end{equation}
 partitioned as the system.
Define
$\widehat{\bV} := \textsf{diag}(\bI,\bV_2,\bI)$.
Then  the reduced system
\begin{equation}  \label{eqn:romind12case}
\Er \dot\bx_r = (\Jr - \Rr) \bx_r + \cbBr \bu(t),~~\by_r = \cbCr \bx_r + \Dr,
\end{equation}
%
%
\begin{align}
\mbox{with}~~ \Er &=  \widehat \bV^T\bE\widehat \bV,~~ \nonumber
\Jr  =  \widehat \bV^T \bJ \widehat \bV,~~ \Rr =  \widehat \bV^T \bR \widehat \bV, \\
\cbBr&= \widehat \bV^T\cbB   = \widehat \bV^T \bB - \widehat \bV^T \bP,~~\cbCr= \bC \widehat \bV = \bB^T \widehat \bV^T + \bP^T \widehat \bV^T,~~~~\mbox{and}~~~\Dr = \bD,
\label{redssind2ind2}
\end{align}
is still \textsf{pH}, matches the polynomial part of the transfer function, and satisfies the tangential interpolation conditions, i.e.,
%
$\bH(\sigma_i) \Rdir_i = \Hr(\sigma_i) \Rdir_i$, for $i=1,2,\ldots,r$.
%
\end{theorem}
\begin{proof}
It follows from the definitions of $\bV$ and $\widehat{\bV}$ that $\mathsf{span}(\bV) \subseteq \mathsf{span}(\widehat{\bV})$. Therefore, the resulting  reduced system  automatically satisfies the interpolation conditions. Since $ \widehat{\bV}$ does not alter the algebraic constraints, the  polynomial part of its transfer function  is still
$ \bD$, matching that of the original model.
The reduced system in \eqref{redssind2ind2} is a \textsf{pHDAE}  as this one-sided projection retains the original  \textsf{pH}  structure.
\end{proof}

\section{Numerical Eperiments}
\label{sec:imp}
The preceding analysis presumed that interpolation points and tangent directions were specified beforehand.
We consider now how one might make choices that generally produce effective approximations with respect to the $\H2$ system measure.

{\bf $\H2$-inspired structure-preserving interpolation:}
between the full model $\bH(s)$ and the reduced model $\bH_r(s)$:
\begin{equation*}
\big\| \bH  - \bH_r\big\|_{\H2} = \left( \frac{1}{2\pi} \int_{-\infty}^\infty \big\| \bH(\imath \omega) - \bH_r(\imath \omega)\big\|_F^2 d\omega \right)^{1/2},
\end{equation*}
where $\imath^2 = -1$ and $\| \bM \|_F$ denote the Frobenius norm of a matrix $\bM$. To have a finite $\H2$ error norm the polynomial parts of $\bH_r(s)$ and $\bH(s)$ need to match. To make this precise, write $\bH(s)=\bG(s) + \bP(s)$ and $\bH_r(s)=\bG_r(s) + \bP_r(s)$ as
where $\bG(s)$ and $\bG_r(s)$ are strictly proper transfer functions, and  $\bP(s)$ and $\bP_r(s)$ are the polynomial parts. Therefore, if $\bH_r(s)$ is the $\H2$-optimal approximation to
$\bH(s)$, then $\bP_r(s) = \bP(s)$ and $\bG_r(s)$ is the  $\H2$-optimal approximation to
$\bG(s)$.  This suggests that one decomposes $\bH(s)$ into its rational and polynomial parts $\bH(s) = \bG(s) + \bP(s)$ and then applies $\H2$ optimal reduction to $\bG(s)$. However, this requires the explicit construction of $\bG(s)$. This problem was resolved in \cite{GugSW13} for \emph{unstructured} index-$1$ and index-$2$ \textsf{DAE}s. On the other hand, for the \textsf{ODE} case,  \cite{GugPBS12} proposed a  \textsf{pH} structure preserving algorithm for minimizing the $\H2$ norm. In this section, we aim to unify these two approaches.

First, we briefly revisit the interpolatory $\H2$ optimality conditions; for details we refer the reader to \cite{GugAB08,BG2017siam,AntBG20} and the references therein.  Since  $\H2$ optimality for the \textsf{DAE} case boils down to optimality for the \textsf{ODE} part, we focus on the latter. Let
$\bG_r(s) = \sum_{i=1}^r \frac{\bc_i \bb_i^T}{s-\lambda_i}$
be the pole-residue decomposition of $\bG_r(s)$. For simplicity we assume simple poles. If $\bG_r(s)$ is an $\H2$ optimal approximation to $\bG(s)$, then
$\bG(-\lambda_i) \bb_i  = \bG_r(-\lambda_i) \bb_i$, $
 \bc_i^T \bG(-\lambda_i) =  \bc_i^T  \bG_r(-\lambda_i)$, and $
 \bc_i^T \bG'(-\lambda_i) \bb_i  = \bc_i^T  \bG'_r(-\lambda_i) \bb_i$.
for $i=1,2,\ldots,r$ where $'$ denotes the derivate with respect to $s$. Therefore an $\H2$ optimal reduced model is a bitangential Hermite interpolant where the interpolation points are the mirror images of its poles and the tangent directions  are the residue directions. Since the optimality conditions depend on the reduced model to be computed, this requires an iterative algorithm, which 
the Iterative Rational Krylov Algorithm \cite{GugAB08}. 
For other  approaches in $\H2$ optimal approximation, see, e.g., \cite{AntBG20,BG2017siam}.

Following \cite{GugPBS12}, in order to preserve the structure, we will satisfy only a subset of these conditions. We will make sure that the interpolation points will be the mirror images of the reduced order poles and enforce  either left or right-tangential interpolation conditions without  the derivate conditions. However, intuitively, one might expect that in the  \textsf{pH}  setting this may not cause too much deviation from true optimality, since this
a Galerkin projection, i.e., $\bZ = \bV$.

Consider the semi-explicit index-$2$ case. 
Starting with an initial selection of  interpolation points $\{\sigma_1,\sigma_2,\ldots,\sigma_r\}$ and the tangent directions $\{ \Rdir_1, \Rdir_2,\dots, \Rdir_r\}$, construct the interpolatory reduced \textsf{pHDAE} $\Hr(s)$ as in Theorem \ref{thm:indtwo1}. Let
$\Hr(s) = \sum_{i=1}^r \frac{\widehat{\Ldir}_i \widehat{\Rdir}_i^T}{s-\lambda_i} + \bD $ be the pole-residue decomposition. Since initially the optimality condition $\sigma_i = -\lambda_i(\Jr - \Rr,\Er)$ is not (generally) satisfied, choose $-\lambda_i(\Jr - \Rr,\Er)$ as the next set of interpolation points and $\widehat{\Rdir}_i$ as the next set of tangent directions. This is repeated until convergence upon which the reduced model $\bH_r(s)$ is not only a structure-preserving  \textsf{pHDAE}, but also satisfies
$\sigma_i = -\lambda_i(\Jr - \Rr,\Er)$.

{\bf Numerical Example:}
We illustrate the discussed procedure 
with 
the incompressible fluid flow model of the  Oseen equations, from \cite[\S 4.1]{HauMM19_ppt}
\begin{align*}
\begin{array}{rclcrclr}
\partial_t v & = & -(a \cdot \nabla)v + \mu \Delta v - \nabla p + f  & ~\mbox{in} ~\Omega \times (0,T],~\quad~
&v & = &0, & \mbox{on}~ \partial \Omega \times (0,T], \\
0 & = & -\mbox{div} v, & ~\mbox{in}~ \Omega \times (0,T], ~\quad~&v & = &v^0, & \mbox{in}~ \Omega \times {0},
\end{array}
\end{align*}
where $v$ and $p$ are the velocity and pressure variables, $\mu>0$ is the viscosity, and
$\Omega = (0,1)^2$ with boundary $\partial \Omega$.  $f$ is an externally imposed body force that for simplicity is assumed to be separable: $f(x,t)=b(x)u(t)$.
 A finite-difference discretization on a staggered rectangular grid
 leads to a single-input/single-output index-$2$ \textsf{pHDAE}  of the form \eqref{eqn:ind2case1}, see \cite{HauMM19_ppt}. In our model, we used
 a uniform grid with $50$ grid points yielding a descriptor system\textsf{pHDAE}  of order $n=7399$, of which  $n_1=4900$ degrees of freedom are for velocity and $n_2=2499$ for pressure. We initialized our approach 
 to reduce the order of the tranfer function to $r=1,2,\ldots,10$ with logarithmically spaced  interpolation points in the interval $[10^{-2},10^4]$. For every $r$ value, we compute
 $\| \bH - \Hr\|_\infty/ \| \bH\|_\infty$ where $\| \bH\|_\infty = \sup_{\omega \in \IR} \mid \bH(\imath \omega)\mid$. The results in Figure \ref{fig:phDAEexmp} show that the reduced transfer function accurately approximates the original one; for the  reduction
 to order $r=10$, the relative error is around $10^{-5}$, illustrating the success of the proposed framework.

\begin{figure}
\centering
\begin{tabular}{cc}
 {\includegraphics[width=5.75cm]{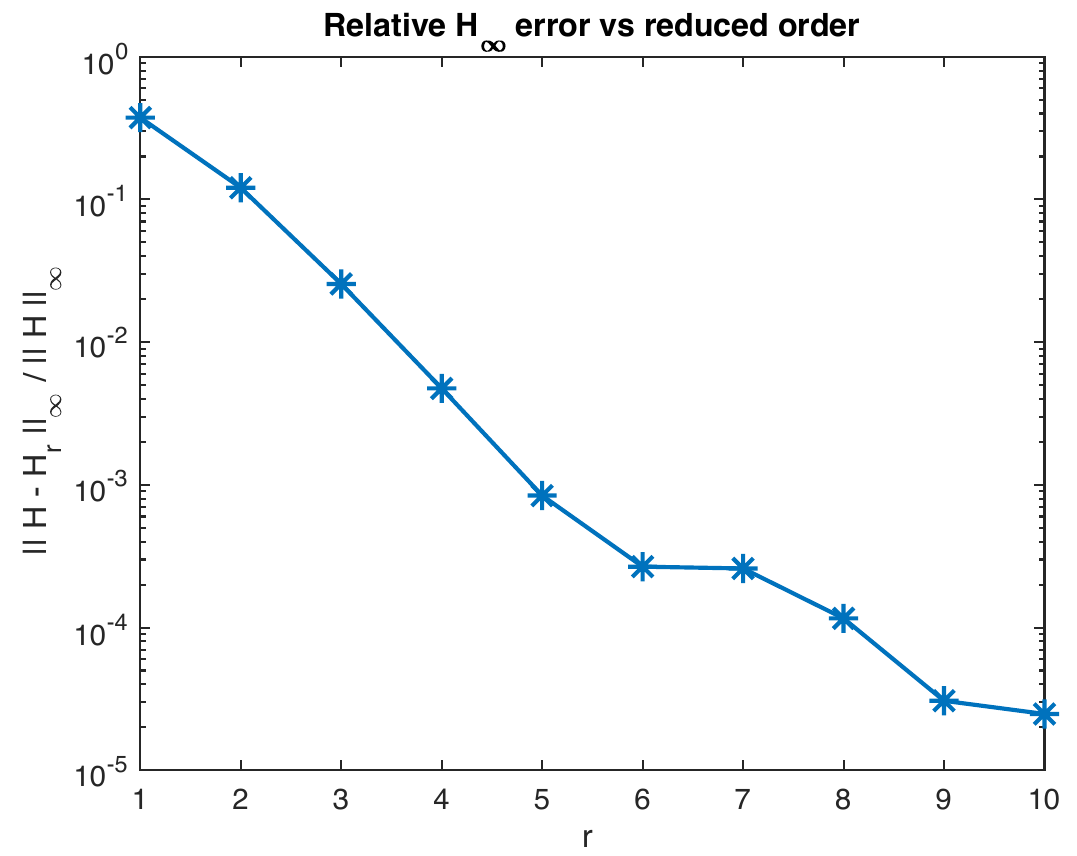}}
\end{tabular}
\caption{Evolution of the model reduction error for Oseen example as $r$ varies}
\label{fig:phDAEexmp}
\end{figure}
\section{Conclusion}

We have presented interpolatory model reduction methods for several classes of large scale linear time-invariant port-Hamiltonian differential-algebraic systems.  We have shown how constraints are  represented in the transfer function so that the polynomial part can be preserved with interpolatory methods. We have  illustrated the results with numerical example from flow control.

\textbf{Acknowledgment}
The works of Beattie and Gugercin were supported in parts by NSF through Grant DMS-1819110. The work of Mehrmann was supported by Deutsche Forschungsgemeinschaft  through  CRC 910 within project A02. Parts of this work were completed while Beattie and Gugercin visited TU Berlin, for which Gugercin acknowledges support through  CRC 910 and Beattie acknowledges support through  CRC TRR 154 as well as DFG Research Training Group DAEDALUS.
\bibliographystyle{plain}
\bibliography{PHDAEPaper}
\end{document}